\documentclass[11pt]{amsart}

\usepackage{amsmath,amssymb,latexsym,soul,cite}

\usepackage{color}
\usepackage{graphicx}
\usepackage{subfigure}
\usepackage{pgf,tikz}
\usetikzlibrary{shapes.geometric}
\usetikzlibrary{decorations.pathreplacing,angles,quotes}
\usetikzlibrary{arrows}
\usepackage{mathrsfs}
\usepackage{xfrac}

\usepackage{enumitem}
\usepackage{booktabs}

\definecolor{grey}{rgb}{.7,.7,.7}

\usepackage[hyphens]{url}
\usepackage[
colorlinks=true,
urlcolor=teal,
citecolor=red,
linkcolor=teal,
linktocpage,
pdfpagelabels,
bookmarksnumbered,
bookmarksopen,
breaklinks=true]
{hyperref}
\usepackage[
capitalize, 
nameinlink,
noabbrev]
{cleveref}

\theoremstyle{plain}
	\newtheorem{thm}{Theorem}[section]
	
	\newtheorem{prop}[thm]{Proposition}
	\newtheorem{cor}[thm]{Corollary}
	
\theoremstyle{definition}
	
	\newtheorem{rem}[thm]{Remark}
\numberwithin{equation}{section}

\newcommand{\R}{{\mathbb{R}}}
\newcommand{\N}{{\mathbb{N}}}

\newcommand{\g}{{\gamma}}

\newcommand{\fr}{{\partial}}

\newcommand{\calH}{\mathcal{H}}
\newcommand{\J}{\mathcal{J}}
\newcommand{\K}{\mathcal{K}}

\newcommand{\de}{\mathrm{d}}
\DeclareMathOperator{\per}{Per}
\DeclareMathOperator{\diam}{diam}
\DeclareMathOperator{\side}{s}
\DeclareMathOperator{\ap}{a}
\DeclareMathOperator{\circum}{R}

\title[Optimization of the anisotropic Cheeger constant]{Optimization of the anisotropic Cheeger constant with respect to the anisotropy}

\author[E.~Parini]{Enea Parini}
\author[G.~Saracco]{Giorgio Saracco}

\address[E. Parini]{Aix Marseille Univ, CNRS, Centrale Marseille, I2M, 39 Rue Fr\'ed\'eric Joliot Curie, F-13453 Marseille CEDEX 13, FRANCE}
\email{enea.parini@univ-amu.fr}

\address[G. Saracco]{Dipartimento di Matematica, Universit\`a di Trento, via Sommarive 14, I-38123 Povo (TN), ITALY} \email{giorgio.saracco@unitn.it}


\thanks{E.~P.~is partially supported by the project ANR-18-CE40-0013 SHAPO financed by the French Agence Nationale de la Recherche (ANR). G.~S.~is a member of the INdAM--GNAMPA group and has been partially supported by the INdAM--GNAMPA 2022 Project ``Stime ottimali per alcuni funzionali di forma'' and the UNITN Starting Grant Giovani 2021 ``WeiCAp''. This work has been started during a visit of E.~P.~in Pavia supported by the Blue Sky Research Project.}

\begin{document}

\definecolor{eqeqeq}{rgb}{0.85,0.85,0.85}

\begin{abstract}
Given an open, bounded set $\Omega$ in $\R^N$, we consider the minimization of the anisotropic Cheeger constant $h_K(\Omega)$ with respect to the anisotropy $K$, under a volume constraint on the associated unit ball. In the planar case, under the assumption that $K$ is a convex, centrally symmetric body, we prove the existence of a minimizer. Moreover, if $\Omega$ is a ball, we show that the optimal anisotropy $K$ is not a ball and that, among all regular polygons, the square provides the minimal value.
\end{abstract}

 \hspace{-3cm}
 {
 \begin{minipage}[t]{0.6\linewidth}
 \begin{scriptsize}
 \vspace{-3cm}
This is a pre-print of an article published in \emph{Canad. Math. Bull.}. The final authenticated version is available online at: \href{https://doi.org/10.4153/S0008439523000152}{https://doi.org/10.4153/S0008439523000152}
 \end{scriptsize}
\end{minipage} 
}

\maketitle


\section{Introduction}

Given an open, bounded set $\Omega$ in $\mathbb{R}^N$, the \emph{Cheeger problem} amounts to finding sets attaining the \emph{Cheeger constant} $h(\Omega)$, which is defined as
\begin{equation}\label{eq:h_Cheeger}
h(\Omega) := \inf \left\{\,\frac{\per(E)}{|E|}\,:\, E\subset \Omega, \, |E|>0 \,\right\},
\end{equation}
where $\per(E)$ denotes the variational perimeter of $E$, and $|E|$ its $N$-dimensional Lebesgue measure. This problem has been first introduced in the Riemannian setting in~\cite{Che70}, and it has been deeply studied ever since, given its many applications in different problems. We refer the interested reader to the two surveys~\cite{Leo15, Par11}. The computation of the constant and the geometric characterization of minimizers is by now well understood in the two-dimensional case for convex sets~\cite{KL06}, strips~\cite{LP16} and more general sets~\cite{LNS17, LS20}. 

A generalization of this problem can be given in the anisotropic Euclidean space, that is, the Euclidean space endowed with a norm induced by a convex, central-symmetric set. In turns, this underlying anisotropic metric induces a notion of anisotropic perimeter, that can be used to define an anisotropic analogous of~\eqref{eq:h_Cheeger}.

More precisely, let $\K_N$ be the class of (non-empty) open, bounded and centrally symmetric (with respect to the origin) convex sets in $\R^N$. Given any $K\in \K_N$, its \emph{polar set} $K^\circ$ is defined as
\[ 
K^\circ := \{ x \in \R^N\,|\,x \cdot y < 1 \text{ for every } y \in K \}.
\]
For every $K \in \K_N$, its polar set $K^\circ$ belongs to $\K_N$, and $(K^\circ)^\circ = K$ (we refer to~\cite{Sch14book} for these standard facts). In particular, given any such $K$, the map $\Phi_K^\circ$ defined as
%
\begin{equation}\label{eq:dual_norm}
\Phi_K^\circ(x) := \sup \{x\cdot y\colon y\in K\}\,,
\end{equation}
is a norm over $\mathbb{R}^N$, called \emph{polar norm} of $K$. By definition,
\[ 
K^\circ = \{ x \in \R^N \,|\,\Phi_K^\circ(x) < 1\},
\]
that is, $K^\circ$ is the unit ball with respect to the metric induced by $\Phi^\circ$.

By means of the polar norm, it is possible to define an anisotropic perimeter for any Borel set $E$ as
\begin{equation}\label{eq:per_K}
\per_K(E):= \int_{\fr^* E} \Phi_K^\circ(\nu_E(x))\, \de\calH^{N-1}(x),
\end{equation}
where $\fr^*E$ denotes the reduced boundary of $E$, see~\cite{Mag12book}.

Notice that such a perimeter is invariant under translation but, in general, not under the effect of the rotation group $\mathrm{SO}(N)$. The \emph{anisotropic isoperimetric inequality}~\cite{Tay78} states that, among sets $E$ of fixed volume, the \emph{unique} anisotropic perimeter minimizer (up to translations) is given by a dilation of the convex set $K$, which is called \emph{Wulff shape} associated to $\Phi^\circ$, i.e.,
\begin{equation}\label{eq:iso_ani}
\per_K(E) 
\ge 
\per_K(K_E)=N|K|^{\frac 1N}|K_E|^{\frac{N-1}{N}},
\end{equation}
where $K_E$ is the dilation of $K$ such that $|E|=|K_E|$. 

With this notion of anisotropic perimeter~\eqref{eq:per_K} at our disposal, we can define the \emph{$K$-Cheeger constant} of a set $\Omega$ analogously to~\eqref{eq:h_Cheeger} as
\begin{equation}\label{eq:K-C}
h_K(\Omega) := \inf \left \{ \frac{\per_K(E)}{|E|} \colon E\subset \Omega\,, |E|>0 \right\}.
\end{equation}
Sets attaining the infimum are called \emph{$K$-Cheeger sets} of $\Omega$, and it is well-known that they exist for any $\Omega$ regular enough with finite measure~\cite{KN08, Sar18}. The constant $h_K(\Omega)$, whenever $\Omega$ is Lipschitz regular, can be thought of as the first eigenvalue of the anisotropic $1$-Laplacian~\cite{KN08}; it is related to anisotropic capillarity problems~\cite{ABT15}, and it is relevant for applications to image reconstruction~\cite{CFM09}.

A usual problem for shape functionals is to determine which shapes $\Omega$ minimize the functional $\Omega \mapsto h_K(\Omega)$ under a volume constraint on $\Omega$. The anisotropic isoperimetric inequality~\eqref{eq:iso_ani} immediately implies that the minimizing shape is a dilation of $K$ itself, and, in particular, we have
\begin{equation}\label{eq:fixed_K}
\inf\left\{\,h_K(\Omega)\,:\, |\Omega|=1 \,\right\} = \frac{\per_K(K)}{|K|} = N.
\end{equation}
While in~\eqref{eq:fixed_K} the set $K\in \K_N$ providing the metric is fixed and one minimizes among $\Omega$, in the present paper we fix $\Omega$ and want to minimize among the metrics $K$ --- under suitable constraints.

There are two possible reasonable choices for the volume constraint: either on the volume of the Wulff shape $|K|$, or on the volume of the unit ball $|K^\circ|$, leading to the study of the two (scaling invariant) functionals
\begin{align}
\mathcal{F}_\Omega[K] &:= h_K(\Omega) |K|^{-\frac{1}{N}}\,,\label{eq:funct_vincolo_K}\\
\J_\Omega[K] &:= h_K(\Omega) |K^\circ|^{\frac{1}{N}}\,.\label{eq:funct}
\end{align}
%
Since we are interested in the metric, it feels more natural to impose a constraint on the volume of the unit ball $K^\circ$, that is, to consider~\eqref{eq:funct}. It is noteworthy that not only it is the more natural choice, but also the minimization of~\eqref{eq:funct_vincolo_K} is a trivial task, whenever $\Omega$ is fixed in $\K_N$. Indeed, by the anisotropic isoperimetric inequality~\eqref{eq:iso_ani}, for every $E \subset \Omega$ one has
\[ 
\per_K(E)|K|^{-\frac{1}{N}} \geq \per_K(K_E) |K|^{-\frac{1}{N}} = N|K_E|^\frac{N-1}{N} = N|E|^\frac{N-1}{N}\,,
\]
with equality holding if and only if $E$ equals $K_E$ up to a translation. Therefore,
\[
\frac{\per_K(E)}{|E|}|K|^{-\frac{1}{N}} \geq N|E|^{-\frac{1}{N}} \geq N|\Omega|^{-\frac{1}{N}}.
\]
Passing to the infimum on all $E \subset \Omega$ we obtain
\begin{equation}
h_K(\Omega) |K|^{-\frac{1}{N}} \geq N|\Omega|^{-\frac{1}{N}}\,,
\end{equation}
for all possible $K\in \K_N$. In particular, equality holds if and only if $K$ coincides with $\Omega$ (up to translations and dilations), and $\Omega$ is the associated Cheeger set. Summing up, if $\Omega \in \K_N$, one has
\begin{equation}\label{eq:minimization_F}
\inf_{K\in \K_N} \mathcal{F}_\Omega[K] = \mathcal{F}_\Omega[\Omega] = N|\Omega|^{-\frac{1}{N}},
\end{equation}
and $K=\Omega$ is the unique minimizer.

On the other hand, the minimization of~\eqref{eq:funct} is far from being trivial. Notice that we can immediately rewrite $\J_\Omega$ as
\begin{equation} \label{eq:rewrite}
\J_\Omega[K] = \mathcal{F}_\Omega[K] (|K||K^\circ|)^{\frac1N}.
\end{equation}
We have already observed that the first factor of the above product is minimized in $\K_N$ by the choice $K=\Omega$, provided that $\Omega\in \K_N$. The second factor is known as the \emph{Mahler volume} of $K$, a shape functional which is invariant under invertible affine transformations, and which is well known to be maximized by balls (in general, ellipsoids~\cite{San49}). Regarding its minimization, it is conjectured to be minimized by (affine transformations of) hypercubes (in general, Hanner's polytopes). This holds true in dimension $N=2$, see~\cite{Mah39}, and under some additional assumptions in higher dimension, see~\cite{Rei86}. We refer to the expository article~\cite[Chap.~3.3]{Tao08} for further details. Therefore, when $N=2$ and $\Omega$ is a parallelogram, the minimization is trivial, while otherwise a competition between the two terms arises.

Similar problems of finding the best metric have been studied, for instance, for the eigenvalues of the Laplace--Beltrami operator on the sphere $\mathbb{S}^2$, see~\cite{Her70, Nad02, NS17, KNPP21}. In the same spirit, one could consider minimization problems with respect to the underlying metric for different shape functionals, e.g., depending on the eigenvalues of the anisotropic $p$-Laplacian, the $K$-capacity or $K$-torsion. We leave these problems as further research directions, as they have an intrinsic higher difficulty due to the fact that they involve non-geometrical quantities.

In the present paper, we focus on the planar case. Moreover, we will always suppose that the set $\Omega \subset \R^2$ is convex, since in this case we can exploit the structure of $K$-Cheeger sets granted by~\cite[Thm.~5.1]{KN08}. In \cref{sec:results} we prove that there exist minimizers of $\J_\Omega$, while its supremum is $+\infty$. 
In \cref{sec:example} we show by means of an example that the problem is indeed nontrivial: in the case $\Omega=B$, we prove that the square is the shape that yields the lowest energy among all regular $n$-gons, and we conjecture it to be the best shape among all possible centrally symmetric anisotropies.

\section{Preliminary results}

%

\subsection{The anisotropic Cheeger problem}

Let $K \in \K_N$, and let $\Omega \subset \R^N$ be an open, bounded (non-empty) set. A $K$-Cheeger set is a non-negligible measurable set $C_K \subset \Omega$ such that
\[ 
h_K(\Omega) = \frac{\per_K(C_K)}{|C_K|}.
\]
Whenever $\Omega$ is a convex, two-dimensional set, there exists a unique $K$-Cheeger set $C_K \subset \Omega$. Moreover, it is possible to give a complete geometrical characterization of $C_K$, which we recall in \cref{thm:struttura}. This has been first proved in the isotropic Euclidean setting~\cite{KL06} and later extended to the anisotropic Euclidean setting in~\cite{KN08}. For the ease of the reader, before stating the result, we recall the definition of Minkowski addition and difference between two sets $E$ and $F$. Given any $x\in \R^2$, if we denote by $E+x$ the translation of the set $E$ by $x$, we have
\begin{align*}
	E\oplus F &:= \cup_{x\in F} (E+x),\\
	E\ominus F &:= \cap_{x\in F} (E-x).
\end{align*}
When $F$ is chosen as a ball $B$, the Minkowski addition $E\oplus B$ can be thought as an outward regularization of the set $E$, and the Minkowski difference $E\ominus B$ as an inward regularization. We remark that, in general, these operations do not commute and one only has the set inclusion $(E\ominus B) \oplus B \subset E$.

\begin{thm}\cite[Thm.~5.1]{KN08}\label{thm:struttura}
	Let $\Omega \subset \R^2$ be an open, bounded, convex set. Then, there exists a unique  $K$-Cheeger set $C_K$ of $\Omega$. Moreover, $C_K$ is convex and we have
	\[
	C_K = \Omega^\rho \oplus \rho K\,,
	\]
	where $\Omega^\rho := \Omega \ominus \rho K$ and $\rho$ is the inverse of the $K$-Cheeger constant. Moreover, $\rho$ is the unique value such that $|\Omega^{\rho}|=\rho^2 |K|$.
\end{thm}


\subsection{The Mahler volume}

Let $K \subset \R^N$ be a convex set. The \emph{Mahler volume} of $K$ is the quantity
\[ V(K) := |K||K^\circ|,\] 
where $K^\circ$ is the polar set of $K$. It can be proven that the Mahler volume is invariant under invertible affine transformations. The following result holds true.

\begin{prop} \label{mahler}
Let $K$ be any convex set in $\mathcal{K}_2$, and let $V(K)$ be its Mahler volume. Then
\[ 8= V(Q) \leq V(K) \leq V(B) = \pi^2,\]
where $Q$ is a square, and $B$ is a ball.
\end{prop}

The upper bound is known as \emph{Blaschke--Santal\'{o} inequality}, whose proof can be found in~\cite[Sect.~10.5]{Sch14book}. The lower bound was proven in~\cite{Mah39}, and an accessible proof can be found in~\cite{Tho07}. 


\subsection{Uniform convergence of polar norms}

The following result provides a link between convergence of the metrics in the Hausdorff distance of convex sets, and local uniform convergence of the associated polar norms.


\begin{prop} \label{uniformcontinuity}
	Let $K \in \K_N$ and $\{K_n\}_{n\in \N} \subset \K_N$ be a sequence such that $K_n\to K$ in the Hausdorff topology. Then, $\Phi_{K_n}^\circ \to \Phi_{K}^\circ$ locally uniformly. 
\end{prop}
\begin{proof}
	Let $\varepsilon > 0$, for $n$ sufficiently big, $(1-\varepsilon)K \subset K_n \subset (1+\varepsilon)K$. If $y \in \mathbb{S}^{N-1}$, it holds
	\[ 
	\Phi_{(1-\varepsilon)K}^\circ(y) \leq \Phi_{K_n}^\circ(y) \leq \Phi_{(1+\varepsilon)K}^\circ(y)
	\]
	which implies
	\[ 
	(1-\varepsilon)\Phi_{K}^\circ(y) \leq \Phi_{K_n}^\circ(y) \leq (1+\varepsilon)\Phi_{K}^\circ(y). 
	\]
	By the $1$-homogeneity of $\Phi_{K}^\circ$ and $\Phi_{K_n}^\circ$ we obtain the claim.
\end{proof}

\section{Main results}\label{sec:results}

Throughout this section we will restrict to the two-dimensional case, and we will suppose that $\Omega$ is an open, bounded, convex planar set. We consider the functional $\J_\Omega$ introduced in~\eqref{eq:funct}, that is,
\[ 
\J_\Omega[K] := h_K(\Omega) |K^\circ|^{\frac{1}{2}},
\]
where the multiplicative factor $|K^\circ|^{\frac12}$ appears in order to make it scale invariant.

\begin{prop} \label{maximization}
Let $\Omega \subset \R^2$ be an open, bounded, convex set. Let $V>0$, and let $\{K_n\}_{n \in \N} \subset \K_2$ be a sequence such that $|K_n|=V$ for every $n \in \N$, with $\diam(K_n) \to +\infty$. Then,
\[
 \lim_{n \to +\infty} \J_\Omega[K_n] = +\infty.
\]
\end{prop}
\begin{proof}
By \cref{thm:struttura} the Cheeger set $C_{K_n}$ of $\Omega$ associated to the anisotropy $K_n$ is given by 
\[
C_{K_n} = \bigcup \rho_n K_n, 
\]
where the union is taken among all dilations of the Wulff shape $K_n$ by $\rho_n$ that are contained in $\Omega$. Since $\diam(K_n) \to +\infty$ as $n \to +\infty$, it must necessarily hold $\rho_n \to 0$. Recalling that $h_{K_n}(\Omega)=\rho_n^{-1}$, one infers that $h_{K_n}(\Omega) \to +\infty$ as $n \to +\infty$. By the reverse Mahler inequality contained in \cref{mahler}, the quantities $|K_n^\circ|$ are uniformly bounded from below. Therefore, $\lim_{n \to +\infty} \J_\Omega[K_n] = +\infty$.
\end{proof}

\begin{cor}
Let $\Omega \subset \R^2$ be an open, bounded, convex set.  Then, \[ \sup_{K \in \K_2} \J_\Omega[K] = +\infty.\]

\end{cor}

We now prove that our shape functional $K\mapsto \J_\Omega[K]$ has a minimizer.

\begin{prop}\label{minimization}
There exists $\hat{K} \in \K_2$ such that
\[ 
\J_\Omega[\hat{K}] = \min_{K \in \K_2} \J_\Omega[K].
\]
\end{prop}
\begin{proof}
Let $\{K_n\}_{n \in \N}$ be a minimizing sequence. Without loss of generality, we can suppose that $|K_n|=V$ for some $V>0$. By Mahler's inequality, the volumes of the polar sets $K_n^\circ$ are uniformly bounded from above. If $\diam(K_n) \to +\infty$ as $n \to +\infty$, by \cref{maximization} it would hold $\J_\Omega[K_n] \to +\infty$, a contradiction. Therefore, $\diam(K_n)$ is uniformly bounded. By Blaschke's Selection Theorem \cite[Thm.~1.8.7]{Sch14book}, there exists $\hat{K} \in \K_2$ such that $K_n \to \hat{K}$ in the Hausdorff topology. By \cref{uniformcontinuity}, 
\[ 
\Phi_{K_n}^\circ \to \Phi_{K}^\circ \quad \text{locally uniformly,}
\]
from which we also infer the $L^1_{\text{loc}}$ convergence of $K^\circ_n$ to $\hat{K}^\circ$, and thus $|\hat{K}^\circ| \le \liminf_n |K^\circ_n|$.
Let $C_{K_n}$ be the Cheeger sets associated to $h_{K_n}(\Omega)$. By \cref{thm:struttura} all these sets are convex, and they are contained in the bounded set $\Omega$. Again by Blaschke's Selection Theorem, up to a subsequence which we do not relabel, there exists a set $\hat{C} \subset \Omega$ that is the Hausdorff limit of the sequence. In particular, $C_{K_n}$ converge to $\hat{C}$ in the sense of characteristic functions. By~\cite[Prop.~2.1]{Neu16} and the local uniform convergence of $\Phi_{K_n}^\circ$ to $\Phi_{K}^\circ$ it holds
\[ 
\per_{\hat{K}}(\hat{C}) \leq \liminf_{n \to +\infty} \per_{K_n}(C_{K_n}).
\]
Moreover, $|C_{K_n}|\to |\hat{C}|$ and $|\hat{C}|>0$ as otherwise, also taking into account that $|K^\circ_n|$ is uniformly bounded from above, the anisotropic isoperimetric inequality would contradict the fact that $K_n$ is a minimizing sequence for $\mathcal{J}_\Omega$. Therefore,
\begin{align*}
h_{\hat{K}}(\Omega) |\hat{K}^\circ|^{\frac{1}{2}} 
&
\leq 
\frac{\per_{\hat{K}}(\hat{C})}{|\hat{C}|}|\hat{K}^\circ|^{\frac{1}{2}} 
\\
&
\leq \liminf_{n \to +\infty} \frac{\per_{K_n}(C_{K_n})}{|C_{K_n}|}|K_n^\circ|^{\frac{1}{2}} 
= 
\liminf_{n \to +\infty} h_{K_n}(\Omega)|K_n^\circ|^{\frac{1}{2}} 
\end{align*}
which means
\[ \J_\Omega[\hat{K}] \leq \liminf_{n \to +\infty} \J_\Omega[K_n].\]
Since $\hat{K} \in \K_2$ and $\{K_n\}_{n\in \N}$ is a minimizing sequence, the claim follows.
\end{proof}

\section{Examples}\label{sec:example}

In this section we provide a few examples in order to highlight how the problem is far from being trivial. We have already made the useful observation that one can rewrite the functional $\J_\Omega[K]$ in terms of $\mathcal{F}_\Omega[K]$ as
\begin{equation}\label{eq:competition}
\J_\Omega[K] = \mathcal{F}_\Omega[K] (|K| |K^\circ|)^{\frac12} = \mathcal{F}_\Omega[K]  V(K)^{\frac12}.
\end{equation}
This easily allows to infer that there cannot be a set $K\in \K_2$ that always minimizes $\J_\Omega[K]$ independently of the choice of $\Omega$.

\subsection{\texorpdfstring{$\Omega$}{Omega} is a parallelogram}

Let $T$ be an invertible affine transformation, and consider the parallelogram given by $T(Q)$, where $Q$ is the unit square in $\R^2$. By \cref{mahler} and the fact that the Mahler volume is invariant with respect to invertible affine transformations,
\[ V(T(Q)) \leq V(K) \]
for every $K \in \K_2$. Moreover as noticed in the introduction, see~\eqref{eq:minimization_F}, $T(Q)$ is the unique minimizer of $\mathcal{F}_{T(Q)}[K]$. Hence,
\[
\J_{T(Q)}[T(Q)] = \mathcal{F}_{T(Q)}[T(Q)] \cdot V(T(Q))^\frac{1}{2} \leq \mathcal{F}_{T(Q)}[K] \cdot V(K)^\frac{1}{2} = \J_{T(Q)}[K]
\]
for every $K \in \K_2$, the strict inequality holding whenever $K \neq T(Q)$.

\subsection{\texorpdfstring{$\Omega$}{Omega} is a ball}
Let $\Omega$ be the Euclidean ball $B$ of unit radius. In this section we show that the minimizing convex body $K\in \K_2$ for $\J_B[K]$ can not be the ball itself. More in general, the square $Q$ provides the lowest energy among all possible regular $n$-gons.

Using~\eqref{eq:competition}, \cref{mahler} and equality~\eqref{eq:minimization_F} we have as benchmark
\[
\inf_{K\in \K_2} \J_{B}[K] \le \mathcal{F}_{B}[B] \cdot V(B)^{\frac 12} = 2\sqrt{\pi}\,.
\]

We let $P^*_n$ be a regular $n$-gon, with $n\ge 4$ even, circumscribed to $B$. With this choice, we have the following
\begin{align}
|P^*_n| &= n\tan(\sfrac \pi n)\,, \qquad &|(P^*_n)^\circ| &= n\sin(\sfrac \pi n) \cos(\sfrac \pi n)\,,\label{eq:area_p_e_polare}
\end{align}
where the first one is a well-known fact of Euclidean geometry, while the second equality comes from~\cite[Cor.~4]{BMMR13}. We remark that the polar body of $P^*_n$ coincides with a rotation and a dilation by $\cos(\sfrac{\pi}{n})$ of $P^*_n$. In particular, we have the following information on the side length $\side(\cdot)$, the apothem $\ap(\cdot)$, and the circumradius $\circum(\cdot)$ of $P^*_n$ and its polar body $(P^*_n)^\circ$
\begin{alignat}{6}
\side(P^*_n)&= 2\tan(\sfrac \pi n)\,,\qquad &\side((P^*_n)^\circ) &= 2\sin(\sfrac \pi n)\,,\nonumber\\
\ap(P^*_n)&= 1\,, &\ap((P^*_n)^\circ) &= \cos(\sfrac \pi n)\,,\label{eq:proprieta_PePolare}\\
\circum(P^*_n)&= \cos(\sfrac \pi n)^{-1}\,, &\circum((P^*_n)^\circ) &= 1\,,\nonumber
\end{alignat}
that we shall use through the next computations.

Given the symmetry of $B$, any choice of $K$ yields the same anisotropic constant of $T(K)$, for any rotation $T\in \mathrm{SO}(2)$. Hence without loss of generality, we can suppose $P^*_n$ to be rotated in such a way that it has two sides parallel to the $y$-axis, and thus its polar body $(P^*_n)^\circ$ has one diagonal on the $x$-axis, see \cref{fig:P_and_polar}. 

\begin{figure}[t]
\begin{tikzpicture}
 \node[regular polygon, 
    regular polygon sides = 6, 
    minimum size = 3.463cm,
    rotate=30,
    draw] (p) at (0,0) {};
    \draw[dotted] (0,0) circle (1.5cm);
    \draw[dashed] (0,0) -- (1.5,0);
    \draw[dashed] (0,0) -- (30:1.732);
     \node[regular polygon, 
    regular polygon sides = 6, 
    minimum size = 3cm,
    draw] (p) at (5,0) {};
    \draw[dotted] (5,0) circle (1.5cm);
    \draw[dashed] (5,0) -- (6.5,0);
    \draw[dashed] (5,0) -- +(30:1.299);
\end{tikzpicture}
\caption{On the left the Wulff shape $P^*_n$ and on the right its polar body $(P^*_n)^\circ$ inducing the metric $\Phi_{P^*_n}^\circ$, for $n=6$. The unit radius disk appears dotted.\label{fig:P_and_polar}}
\end{figure}
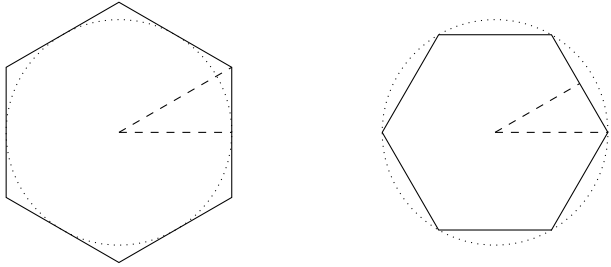
By the symmetry of $B$, of $P^*_n$ and by \cref{thm:struttura}, the boundary of the minimizer is made of $n$ straight sides parallel to those of $P^*_n$ with endpoints on $\fr B$ and $n$ circular and symmetric arcs of $\fr B$, as in \cref{fig:Cheeger_K_ball}. We consider the one-parameter family of competitors $E_x$ that have this particular structure, being the parameter $x=x(n)$ half the length of one of the straight sides. Given the symmetric nature of our setting, we can divide the plane $\R^2$ in $2n$ symmetric circular sectors and compute the area and the anisotropic perimeter of these candidates in just one of these.

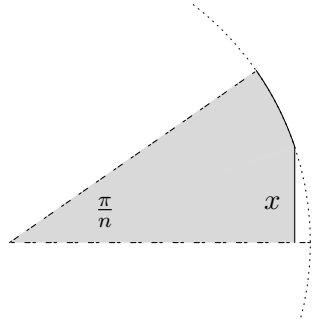
\begin{figure}[t]
\centering
\begin{tikzpicture}[line cap=round,line join=round,>=triangle 45,x=1.0cm,y=1.0cm]
\begin{scope}[yscale=-1,xscale=1, rotate=-90]
\clip(-1.,-0.2) rectangle (3.2,4.2);
\fill[line width=0.pt,color=eqeqeq,fill=eqeqeq,fill opacity=1.0] (0.,0.) -- (1.273,3.791) -- (0.,3.791) -- cycle;
\draw [shift={(0.,0.)},line width=0.pt,color=eqeqeq,fill=eqeqeq,fill opacity=1.0]  (0,0) --  plot[domain=0.962:1.246,variable=\t]({1.*4.*cos(\t r)+0.*4.*sin(\t r)},{0.*4.*cos(\t r)+1.*4.*sin(\t r)}) -- cycle ;
\draw [line width=0.4pt,dash pattern=on 1pt off 1pt on 2pt off 4pt] (0.,4.)-- (0.,0.);
\draw [line width=0.4pt] (0.,3.791)-- (1.273,3.791);
\draw [line width=0.4pt,dash pattern=on 1pt off 1pt on 2pt off 4pt] (0.,0.)-- (2.285,3.282);
\draw [shift={(0.,0.)},line width=0.4pt,dotted]  plot[domain=0.383:1.976,variable=\t]({1.*4.*cos(\t r)+0.*4.*sin(\t r)},{0.*4.*cos(\t r)+1.*4.*sin(\t r)});
\draw [shift={(0.,0.)},line width=0.4pt]  plot[domain=0.962:1.246,variable=\t]({1.*4.*cos(\t r)+0.*4.*sin(\t r)},{0.*4.*cos(\t r)+1.*4.*sin(\t r)});
\end{scope}
\draw (3.25,0.75) node[anchor=north west] {$x$};
\draw (1.,0.75) node[anchor=north west] {$\frac \pi n$};
\end{tikzpicture}
\caption{The shape of the Cheeger set in a sector of width $\sfrac \pi n$ w.r.t.~the anisotropy given by the regular $n$-gon.}\label{fig:Cheeger_K_ball}
\end{figure}
The area in the sector $S_i$ is given by the area of a triangle with base $\sqrt{1-x^2}$ and height $x$ plus the area of a circular sector of radius $1$ and angle $\sfrac \pi n - \arcsin(x)$, yielding
\begin{equation}\label{eq:area_competitor}
|E_{x}| = 2n\left(\frac 12 x \sqrt{1-x^2} + \frac{\pi}{2n}  - \frac 12 \arcsin(x)\right)\,.
\end{equation}
Concerning the perimeter, it is immediate that the Euclidean one is $x$ plus $\sfrac \pi n - \arcsin(x)$, but we should take into account the presence of the anisotropy. The straight side, of Euclidean length $x$, has constant normal given by the horizontal direction $e_1$. By the choices we made, see~\eqref{eq:proprieta_PePolare}, the greatest $y\in \R_+$ such that $ye_1 \in (P^*_n)^\circ$ is $y=1$, thus $\Phi_{P^*_n}^\circ(e_1)=1$. Hence, this straight side has anisotropic perimeter equal to the Euclidean one. Concerning the circular arc, we can parametrize it as $\gamma(\theta) = (\cos\theta, \sin \theta)$ for $\theta \in [\arcsin(x), \sfrac \pi n]$, and thus
\[
\int_{\fr B\cap \gamma} \Phi_{P^*_n}^\circ(\nu_{B}(y)) \de \calH^1(y) = \int_{\arcsin(x)}^{\frac \pi n} \Phi_{P^*_n}^\circ(\theta)\|\dot \g (\theta)\|\de \theta = \int_{\arcsin(x)}^{\frac \pi n} \Phi_{P^*_n}^\circ(\theta)\de \theta\,.
\]
Hence, it is just a matter of computing the anisotropy $\Phi_{P^*_n}^\circ(\theta)$. Given our initial assumptions, the point $e_1$ is a vertex of $(P^*_n)^\circ$, and the $x$-axis splits the interior angle with vertex in $e_1$ in two equal angles of width $\sfrac{(n-2)\pi}{2n}$. Let $y\theta$ belong to the boundary $(P^*_n)^\circ$, for $\theta \in (0, \sfrac \pi n)$. 
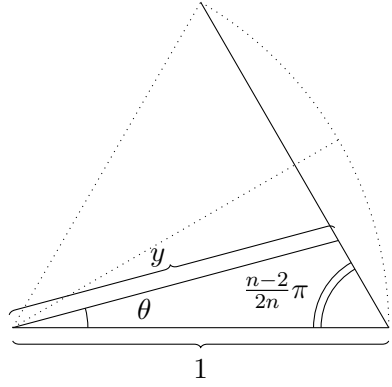
\begin{figure}[t]
\centering
\begin{tikzpicture}[x=5.0cm,y=5.0cm]
\draw[dotted] (0,0)--+(60:1);
\draw (0,0)--(1,0);
\draw[decoration={brace,mirror,raise=5pt},decorate] (0,0) -- node[below=8pt] {$1$} (1,0);
\draw[dotted] (1,0) arc (0:60:1);
\draw[dotted] (0,0)--+(30:1);
\draw (1,0)--+(120:1);
\draw (0,0)--+(15:0.896);
\draw[decoration={brace,raise=5pt},decorate] (0,0) -- node[above=10pt, left] {$y$} +(15:0.896);
\draw (0.2,0) arc (0:15:0.2);
\node at (0.35,0.05) {$\theta$};
\draw (0.8,0) arc (180:120:0.2);
\draw (0.82,0) arc (180:120:0.18);
\node at (0.7,0.1) {$\frac{n-2}{2n}\pi$};
\end{tikzpicture}
\caption{Close up of the unit ball in the metric $\Phi_{P^*_n}^\circ$. The dots individuate the sectors of the Euclidean unit ball $B$ with angles $\sfrac \pi n$ and $\sfrac{2\pi}{n}$.}\label{fig:Phi}
\end{figure}
Using the law of sines, see also \cref{fig:Phi}, we obtain the following equality
\[
y = \frac{\sin\left(\frac{n-2}{2n}\pi \right)}{\sin\left(\pi - \frac{n-2}{2n}\pi - \theta\right)} = \frac{\cos\left(\sfrac{\pi}{n} \right)}{\cos\left(\sfrac{\pi}{n} - \theta\right)},
\]
and, since $\Phi_{P^*_n}^\circ(\theta) = y^{-1}$, we eventually get
\[
\Phi_{P^*_n}^\circ(\theta) = \frac{\cos\left(\sfrac{\pi}{n} - \theta\right)}{\cos\left(\sfrac{\pi}{n} \right)}.
\]
Hence,
\begin{equation}\label{eq:per_competitor}
\begin{split}
\per_{P^*_n}(E_{x}) &= 2n\left(  x + \frac{1}{\cos\left(\frac{\pi}{n} \right)} \int_{\arcsin(x)}^{\frac \pi n} \cos\left(\frac{\pi}{n} - \theta\right) \de \theta \right)\\
&= 2n\left(  x + \frac{\sin\left(\frac{\pi}{n} - \arcsin(x)\right)}{\cos\left(\frac{\pi}{n} \right)}\right).
\end{split}
\end{equation}
By~\eqref{eq:area_competitor} and~\eqref{eq:per_competitor}, it follows that the Cheeger set $C_n$ of $B$, in the metric for which $P^*_n$ is the Wulff shape, coincides with $E_{x}$, for the value $\bar{x}_n$ that minimizes the ratio
\begin{equation}\label{eq:to_minimize}
\frac{2}{\cos\left(\frac{\pi}{n} \right)} \frac{x \cos\left(\frac{\pi}{n} \right) + \sin\left(\frac{\pi}{n} - \arcsin(x)\right)}{ x \sqrt{1-x^2} + \frac{\pi}{n}  - \arcsin(x)}\,.
\end{equation}
Further, recall that \cref{thm:struttura} states that the Cheeger set is the union of $\rho P^*_n$, where $\rho$ is the inverse of the Cheeger constant. Thus, if $\bar{x}_n$ is minimizing~\eqref{eq:to_minimize} we also have the following equality
\[
2\bar{x}_n = \side(\rho P^*_n) = 2 \rho \tan(\sfrac \pi n)\,,
\]
hence the minimizing $\bar{x}_n$ is a solution of
\[
\frac{1}{x}  \tan\left(\frac \pi n\right)=  \frac{2}{\cos\left(\frac{\pi}{n} \right)} \frac{x \cos\left(\frac{\pi}{n} \right) + \sin\left(\frac{\pi}{n} - \arcsin(x)\right)}{ x \sqrt{1-x^2} + \frac{\pi}{n}  - \arcsin(x)},
\]
and
\begin{equation}\label{value_h_P*_n}
h_{P^*_n}(B) = \frac{1}{\bar{x}_n}  \tan\left(\frac \pi n\right).
\end{equation}
At this point, use the trigonometric identity
\[ 
\sin\left(\frac{\pi}{n} - \arcsin(x)\right) = \sin\left(\frac{\pi}{n}\right)\sqrt{1-x^2} - x \cos\left(\frac{\pi}{n}\right) 
\]
to simplify the identity into
\[
\frac{1}{x} =   \frac{2\sqrt{1-x^2}}{ x \sqrt{1-x^2} + \frac{\pi}{n}  - \arcsin(x)}	
\]
and hence $\bar{x}_n$ solves
\begin{equation}\label{eq:other_eq_to_find_h}
\arcsin(x) + x\sqrt{1-x^2} = \frac{\pi}{n}.
\end{equation}
Since the function on the LHS is increasing in $x$, there exists a unique solution, and the function $n \mapsto \bar{x}_n$ is decreasing, so that the maximum value of $\bar{x}_n$ is attained for $n=4$. Unfortunately this equation cannot be solved explicitly for $x$, but its unique solution for each $n$ can be numerically computed: a plot of the map $n \mapsto \bar{x}_n$ is shown in \cref{fig:graph_xn}, and some values are collected in \cref{tab:some_n}.

The lack of an explicit expression for $\bar{x}_n$ prevents us from having an explicit value of $ \J_{B}[P^*_n]$. Nevertheless, we can easily give an upper and a lower bound to $\bar{x}_n$. On the one hand, as a consequence of the isoperimetric inequality, we can estimate $h_{P^*_n}(B)$ from below by $h_{P^*_n}(\rho P^*_n)$, where $\rho$ is such that $|\rho P^*_n| = |B|=\pi$. This gives
\begin{equation}\label{eq:stima_dal_basso}
\frac{1}{\bar{x}_n}  \tan\left(\frac \pi n\right) = h_{P^*_n}(B) \ge h_{P^*_n}(\rho P^*_n) = \frac{2}{\rho}\,.
\end{equation}
By the relations in~\eqref{eq:area_p_e_polare}, we have
\[
\pi = |\rho P^*_n| = \rho^2 |P^*_n| = \rho^2 n \tan\left(\frac \pi n\right)\,,
\]
and solving for $\rho$ and plugging in~\eqref{eq:stima_dal_basso} gives
\[
\bar{x}_n \le \frac{\sqrt{\pi}}{2} \sqrt{\frac{\tan\left(\frac \pi n\right)}{n}}\,.
\]
On the other hand, using as competitor $\cos(\sfrac \pi n) P^*_n$, that is the greatest Wulff shape contained in $B$, by definition of $K$-Cheeger constant we obtain, as a lower bound,
\begin{align*}
\frac{1}{\bar{x}_n}  \tan\left(\frac \pi n\right) 
= h_{P^*_n}(B) 
&\le \frac{\per_{P^*_n}(\cos\left(\frac \pi n\right) P^*_n)}{|\cos\left(\frac \pi n\right) P^*_n|}
\\ 
&= \frac{1}{\cos \left(\frac \pi n\right)} \frac{\per_{P^*_n}(P^*_n)}{|P^*_n|} 
= \frac{2}{\cos \left(\frac \pi n\right)}\,.
\end{align*}
Putting these two inequalities together, we get
\begin{equation}\label{eq:constraint_on_x}
\frac{1}{2} \sin \left(\frac \pi n\right) \le \bar{x}_n \le  \frac{\sqrt{\pi}}{2} \sqrt{\frac{\tan\left(\frac \pi n\right)}{n}}\,.
\end{equation}
This is enough to prove that $\J_B[P^*_n]$ achieves its minimum on $P^*_4$. Indeed, recalling~\eqref{value_h_P*_n} and~\eqref{eq:area_p_e_polare}, one has
\[
\J_{B}[P^*_n] = h_{P^*_n}(B)|(P^*_n)^\circ|^{\frac 12} = \frac{1}{\bar{x}_n}\tan\left(\frac \pi n \right)  \sqrt{n \cos(\sfrac \pi n) \sin(\sfrac \pi n)}
\]
Using the bounds in~\eqref{eq:constraint_on_x}, one has
\[
\frac{2n}{\sqrt{\pi}} \sin\left(\frac{\pi}{n} \right)\le \J_{B}[P^*_n] \le 2 \sqrt{n \tan\left(\frac \pi n\right)}\,.
\]
Both LHS and the RHS converge to $2\sqrt{\pi}$, and, in particular, the LHS is greater than $\J_{B}[P^*_4]$ as soon as $n\ge 12$. In order to conclude, it is enough to check a finite number of values  $\J_{B}[P^*_n]$, corresponding to $n=4, 6, 8, 10$. These, along with some other values, numerically obtained, are reported in \cref{tab:some_n}, along with the value of $\bar{x}_n$ minimizing~\eqref{eq:to_minimize} subject to the constraint~\eqref{eq:constraint_on_x}. Graphs depicting the behavior of $\bar{x}_n$ and $\J_{B}[P^*_n]$ for the first $100$ even numbers $n\ge 4$ are shown in \cref{fig:graph_xn} and \cref{fig:graph_Jn}. We conjecture that $\J_{B}[P^*_n]$ is actually increasing in $n$. 

\begin{figure}[t]
\subfigure[\label{fig:graph_xn}]{
\includegraphics[width=0.45\linewidth, height=5cm,keepaspectratio, trim={0.6cm 0 0 0}, clip]{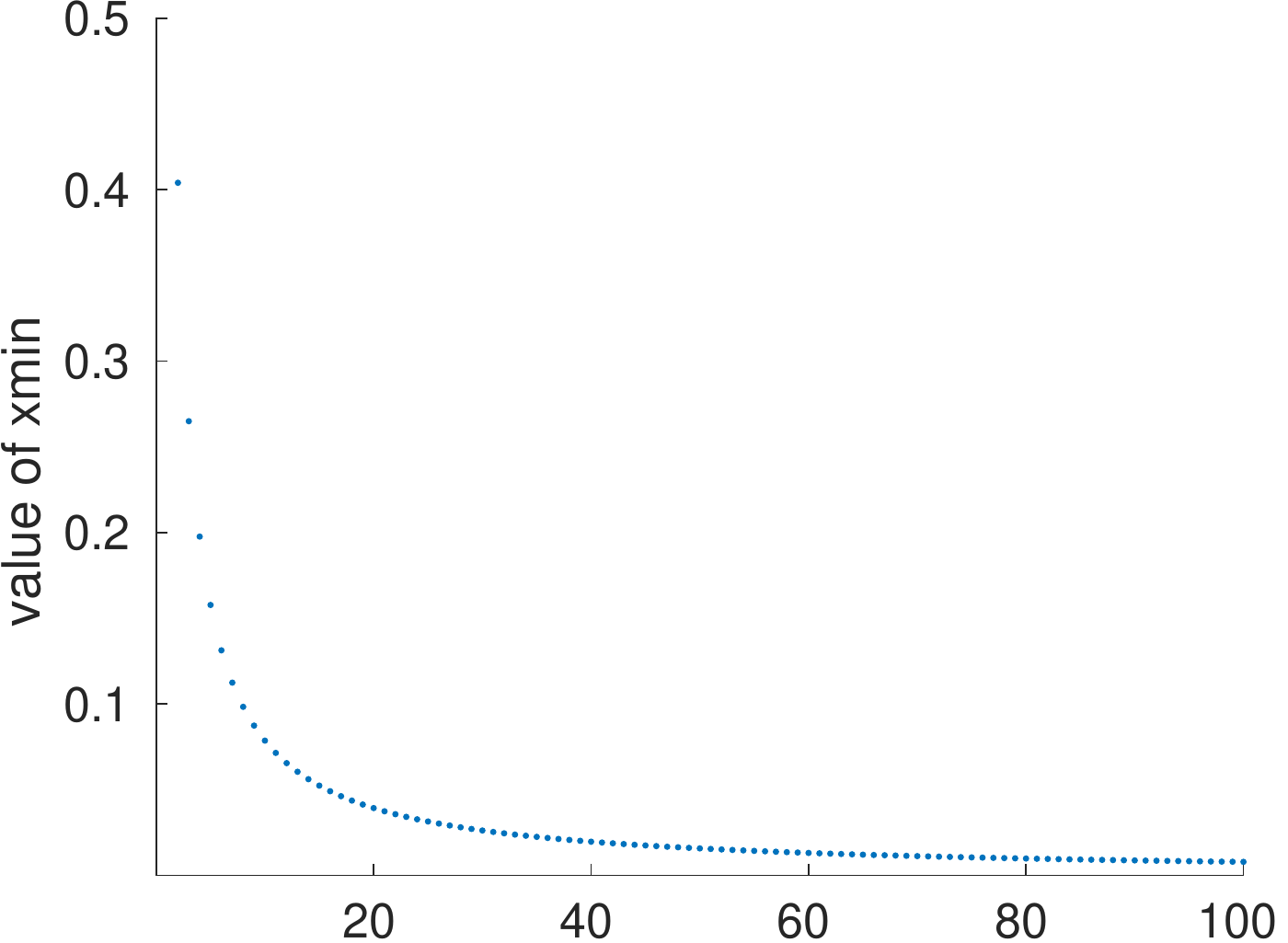}
}
\subfigure[\label{fig:graph_Jn}]{
\includegraphics[width=0.45\linewidth, height=5cm,keepaspectratio, trim={0.6cm 0 0 0}, clip]{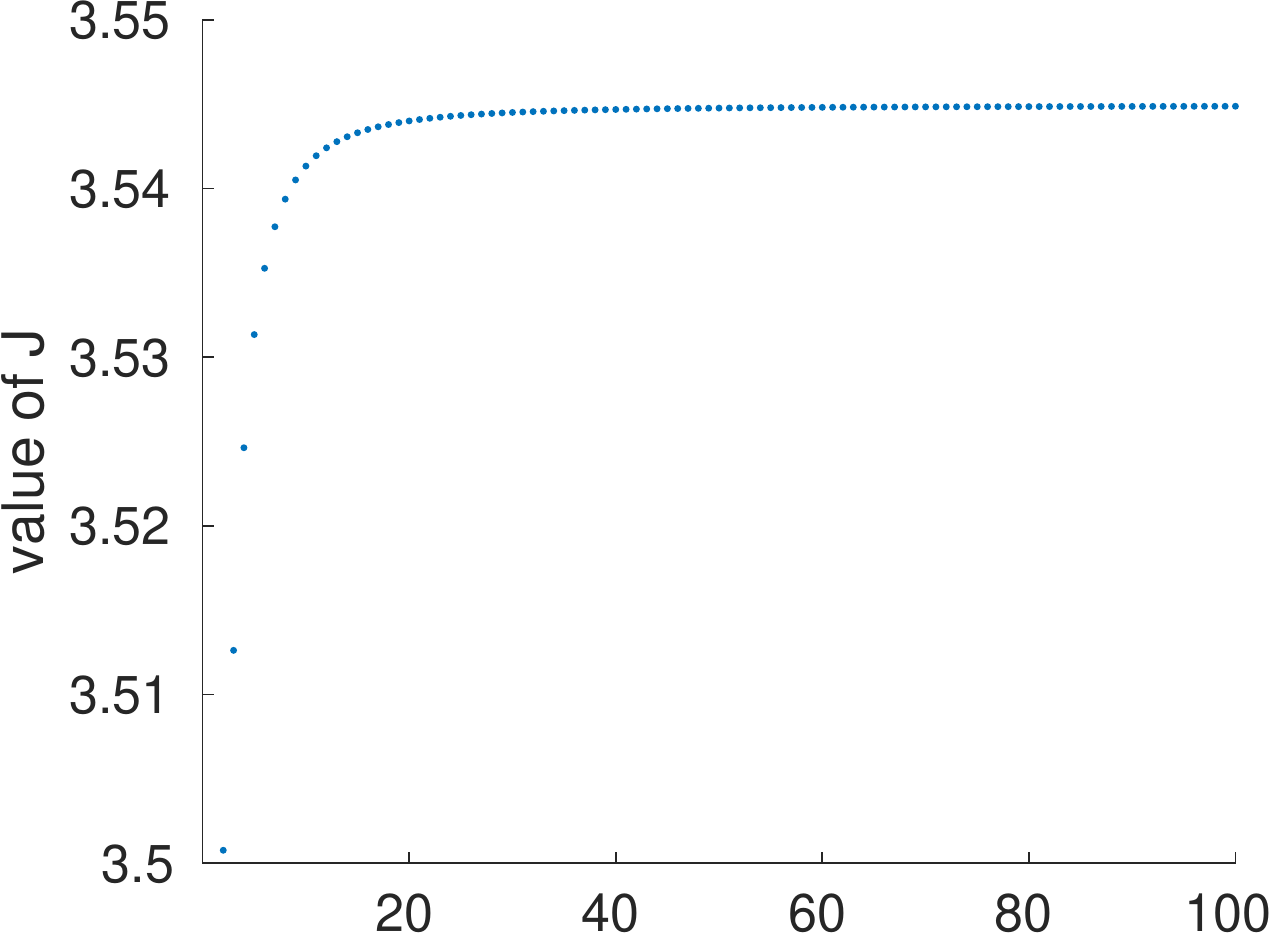}
}
\caption{Graphs of $\bar{x}_n$ (LHS) and of $\mathcal{J}_B[P^*_n]$ (RHS)}
\label{fig:graphs}
\end{figure}
\begin{table}
\centering
\begin{tabular}{lcc c lcc}
& $\bar{x}_n$ & $\J_B[P^*_n]$ & & & $\bar{x}_n$ & $\J_B[P^*_n]$\\
\cmidrule{1-3}\cmidrule{5-7}
$n=4$ & $0.4040\dots$ & $3.5008\dots$ & \qquad & $n=50$ & $0.0315\dots$ & $3.5443\dots$\\
$n=6$ & $0.2649\dots$ & $3.5126\dots$ & & $n=100$ & $0.0158\dots$ & $3.5448\dots$ \\
$n=8$ & $0.1976\dots$& $3.5246\dots$ & & $n=200$ & $0.0079\dots$ & $3.5449\dots$\\
$n=10$ & $0.1578\dots$ & $3.5313\dots$ & & $n=+\infty$ & $0$ & $2\sqrt{\pi}$\\
\bottomrule
\vspace{1pt}
\end{tabular}
\caption{Values of the minimizing half-side $\bar{x}_n$ and of the functional $\J_B[P^*_n]$ for some choices of $n$.}
\label{tab:some_n}
\end{table}

\section{Final remarks and open questions}

In this paper we initiated the study of optimization of anisotropic shape functionals with respect to the anisotropy. We proved that the minimization problem for the functional $\J_\Omega$ is well posed, and we obtained some partial results in the case when $\Omega$ is a ball. We are left with several open questions.
\begin{itemize}
\item Is it possible to show existence of minimizers in the $N$-dimensional case? The difficulty lies in the fact that, apart from the planar case, there is no obvious characterization of $K$-Cheeger sets, which is a key ingredient of the proof of \cref{maximization}.
\item If $\Omega=B$, is it true that $\J_B[P^*_n] \le \J_B[P]$ for any $P\in \mathcal{P}_n$? While this assertion might seem reasonable, we need to observe, bearing in mind equation~\eqref{eq:rewrite}, that the Mahler volume $V$ is actually \emph{maximized}, among all polygons with $n$ sides, by the regular one, see also~\cite{Meyerreisner11}.
\end{itemize}	

The functional that has been considered in this paper involves only purely geometrical quantities. In the literature, several other functionals, involving anisotropic differential operators, have been investigated; we can mention, among others, the first eigenvalue of the anisotropic $p$-Laplacian~\cite{KN08}, which is defined, for $p \in (1,+\infty)$, as
\[ \lambda_p^K(\Omega) := \inf_{u \in W^{1,p}_0(\Omega)} \frac{\int_\Omega (\Phi^\circ(|\nabla v|))^p}{\int_\Omega |u|^p}.\]
It would be interesting to carry a similar analysis for the shape optimization problem
\[ \inf_{K \in \K_N} \lambda_p^K(\Omega) |K^\circ|^\frac{p}{N}. \]
However, this problem is considerably more difficult than the one considered in the present work, since non-geometrical quantities are involved. We remark that the results proved here, combined with some well-known inequalities between the first eigenvalue of the anisotropic Dirichlet $p$-Laplacian and the anisotropic Cheeger constant allow us to say something in the two-dimensional case. Namely, we have the following.

\begin{rem}
The $K$-Cheeger constant provides a lower bound to the first eigenvalue of the anisotropic Dirichlet $p$-Laplacian $\lambda^K_p(\Omega)$ (Cheeger's inequality, see~\cite{KN08}). Then, from \cref{maximization} it follows that, for the scaling invariant functional $K \mapsto  \lambda^K_p(\Omega) |K^\circ|^\frac{p}{2}$, has infinite supremum in $\K_2$.
\end{rem}

\begin{rem}
Besides the lower bound to $\lambda^K_p(\Omega)$  provided by Cheeger's inequality~\cite{KN08}, the $K$-Cheeger constant also provides an upper bound to $\lambda^K_p(\Omega)$, known as Buser's inequality or reverse Cheeger's inequality, refer to~\cite{Par17, Bra20, Fto21, DPdBG20}). Hence, from \cref{minimization} it follows that the shape functional $K\mapsto \lambda^K_p(\Omega)  |K^\circ|^\frac{p}{2}$ has non-zero, finite infimum in $\K_2$. 
\end{rem}

\bibliographystyle{plainurl}

\bibliography{Cheeger-Mahler}

\end{document}